\documentclass[12pt, reqno, a4paper]{amsart}
\usepackage{latexsym}
\usepackage{amsmath,amssymb,exercise}
\usepackage{dsfont}
\usepackage{wasysym}

\usepackage[bookmarks=true,hyperindex,pdftex,colorlinks,citecolor=blue]{hyperref}

 \newtheorem{theo}{Theorem}[section]

  \newtheorem{lem}[theo]{Lemma}

\newtheorem{propos}[theo]{Proposition}

\newtheorem{cor}[theo]{Corollary}

\theoremstyle{definition}
 \newtheorem{definition}[theo]{Definition}
 \newtheorem{prob}[theo]{Problem}

\newcommand{\R}{{\mathbb{R}}}

\newcommand{\N}{{\mathbb{N}}}

\newcommand{\eps}{\varepsilon}

\newcommand{\diam}{\mathrm{diam}}

\newcommand{\conv}{{\mathrm{conv}}}

\newcommand{\spn}{\mathrm{span}}

\newcommand{\eins}{{\mathds{1}}}

\newcommand{\bea}{\begin{eqnarray*}}
\newcommand{\eea}{\end{eqnarray*}}
\newcommand{\beq}{\begin{eqnarray}}
\newcommand{\eeq}{\end{eqnarray}}

\newcommand{\sign}{{\mathrm{sign}}}
\newcommand{\card}{{\mathrm{card}}}

\newcommand{\hausd}{\rho_H}

\renewcommand{\leq}{\leqslant}
\renewcommand{\geq}{\geqslant}
\renewcommand{\le}{\leqslant}
\renewcommand{\ge}{\geqslant}

\numberwithin{equation}{section}

\begin{document}
\title[SLLN for Random Sets in Banach spaces]{Strong Law of Large Numbers for Random Sets in Banach spaces}

\author[Kadets]{V. Kadets}

\address[Kadets]{ \href{http://orcid.org/0000-0002-5606-2679}{ORCID: \texttt{0000-0002-5606-2679}} {School of Mathematical Sciences, Holon Institute of Technology (Israel)}}
\email{kadetsv@hit.ac.il}

\author[Zavarzina]{O. Zavarzina}
\address[Zavarzina]{\href{http://orcid.org/0000-0002-5731-6343}{ORCID: \texttt{0000-0002-5731-6343}} School of Mathematics and Informatics, V.N. Karazin Kharkiv National University (Ukraine)}
\email{olesia.zavarzina@yahoo.com}


\begin{abstract}
The Strong Law of Large Numbers (SLLN) for random variables or random vectors with different mathematical expectations easily reduces by means of shifts to SLLN for random variables or random vectors whose mathematical expectations are equal to zero. The situation changes for random sets, where shifts cannot reduce sets of more than one point to the set $\{0\}$. We study effects that appear because of this difference.
\end{abstract}

\maketitle
\section{Introduction}

There are three forms of the Strong Law of Large Numbers (SLLN in short) that are evidently equivalent for random vectors or random variables.
Namely, let us say that \\
-- a Banach space $X$ enjoys \emph{the full form of SLLN} if for every probability space  $(\Omega, \Sigma, P)$ and for every uniformly bounded independent sequence of random vectors $\Psi_n: \Omega \to X$ with mathematical expectation  $E(\Psi_n) = x_n$ (in the sense of Bochner integral) one has
$$
\left\|\frac{1}{n}\sum\limits_{i=1}^{n} \Psi_i(\omega) - \frac{1}{n}\sum\limits_{i=1}^{n} x_i\right\| \xrightarrow[n \to \infty]{a.e.} 0,
$$
--  a Banach space $X$ enjoys \emph{the intermediate form of SLLN}  if for every probability space  $(\Omega, \Sigma, P)$, for every $x \in X$ and for every uniformly bounded independent sequence of random vectors $\Psi_n: \Omega \to X$ that have the same mathematical expectation  $E(\Psi_n) = x$ one has
$$
\left\|\frac{1}{n}\sum\limits_{i=1}^{n} \Psi_i(\omega) - x\right\| \xrightarrow[n \to \infty]{a.e.} 0,
$$
and, finally,  \\
-- a Banach space $X$ enjoys \emph{the reduced form of SLLN}  if for every probability space  $(\Omega, \Sigma, P)$ and for every uniformly bounded independent sequence of random vectors $\Psi_n: \Omega \to X$ with $E(\Psi_n) = 0$ one has
$$
\left\|\frac{1}{n}\sum\limits_{i=1}^{n} \Psi_i(\omega)\right\| \xrightarrow[n \to \infty]{a.e.} 0.
$$
Evidently,  the full form implies the intermediate form, which in its turn implies the reduced form, and one easily deduces the full form from the reduced one by considering the auxiliary random vectors: $\Phi_i = \Psi_i - x_i$.

For this reason, all the three forms of SLLN in a Banach space $X$ are just called SLLN. It is well known that the SLLN in a Banach space $X$ is equivalent to a geometric property of $X$ called \emph{B-convexity} \cite{beck62, beGiWa75, Giesy1966}, and that in fact the convergence remains true if one instead of uniform boundedness of the random vectors demands uniform boundedness of their variances.

The situation changes when one passes from random vectors to random sets (i.e. multifunctions acting from a probability space, for details on multifunctions and random sets see \cite{hess}, \cite{Molc} and references therein).
All the three forms of SLLN make sense for random sets.   The implications \\
\centerline{ (full form)~$\Longrightarrow$~(intermediate form)~$\Longrightarrow$~(reduced form) }\\
remain evident, but the converse reduction as above does not work for random sets because $E\left(\Psi_i - E(\Psi_i)\right) \neq \{0\}$ for random sets (this is the same effect that $A - A \neq \{0\}$ for sets).

The goal of this article is to establish which forms of SLLN are valid for random sets. After introducing in Section \ref{prelim} the necessary preliminaries about mulifunctions and their integration, we show in Section \ref{form1} that the full form of SLLN for random sets is not valid in any infinite-dimensional Banach space even if we restrict ourselves to random sets that take convex finite-dimensional values. After that, we give in Section \ref{form1fd} an independent proof of the full form of SLLN for random sets in the case of finite-dimensional space $X$ (which can also be deduced from known facts, see \cite[Theorem 1.19]{Molc}).

The reduced form of SLLN for random sets is valid in every B-convex space, because zero mathematical expectation of a random set in $X$ implies that this random set reduces to a random vector in $X$. We show this in detail in Section \ref{form3}, where the main difficulty that we overcome is the extension of the result from separable to non-separable spaces.

Finally, Section \ref{form2} is devoted to the intermediate form of SLLN for random sets, where we have only partial results. Namely, we demonstrate that in the case of $E(\Psi_n)$  being one the same finite-dimensional set, the intermediate form of SLLN for random sets is valid in B-convex spaces, but it remains unclear for us whether the condition of finite-dimensionality can be weaken to, say, compactness.


\section{Preliminaries} \label{prelim}
\subsection{Basic facts and notation}
 In our paper, the letter $X$ stands for a Banach space and $X^*$ denotes the dual space.  For an element $a \in X$ and a subset $B\subset X$ we denote by $\rho(a,B)$ the distance between $a$ and $B$, i.e.
$\rho(a,B)=\inf\limits_{b \in B}\Vert a - b \Vert$.

Denote by $b(X)$ ($bc(X)$) the collection of all non-empty bounded (bounded and convex, respectively) subsets of $X$. For $A \in b(X)$ we denote
$$
\|A\| = \sup_{a \in A} \|a\|.
$$

We will call the \textit{one-sided Hausdorff distance} between
sets $A, B \in b(X)$ the number $\bar{\rho}_H(A,B) = \sup\limits_{b\in B}\rho(b,A)$. This definition means that, for given $r \ge 0$, $\bar{\rho}_H(A,B) \le r$ if and only if for every $b \in B$ and for every $\eps > 0$ there is $a \in A$ with $\|a - b\| < r + \eps$.

Let us remark that $\bar{\rho}_H(A,B)$ is decreasing in the first variable in the following sense: if $A, \widetilde A , B \in b(X)$ and $A \subset \widetilde A$, then
$\bar{\rho}_H(\widetilde A,B) \le \bar{\rho}_H(A,B)$.

The \textit{Hausdorff distance} between $A$ and $B$ is the number
$$
\rho_H(A,B) = \max \{ \bar{\rho}_H(A,B), \bar{\rho}_H(B,A)\}.
$$
 With $\conv B$ and $\spn B$ we denote the convex hull and the closed linear hull of the set $B$ respectively, $\diam B$ denotes the diameter of $B$, $\card(M)$ stays for the number of elements of the finite set $M$, for a subset $\Delta$ of a set $\Omega$ we denote $\mathds{1}_{\Delta}$ the characteristic function of $\Delta$, and $\ker P$ stands for the kernel of an operator $P$.

\subsection{R{\aa}dstr\"om embedding and Bochner integral for multifunctions}

The collection of sets $bc(X)$ is a cone with respect to the natural operations of Minkowski sum
$$
A+B=\{a + b: a \in A, b \in B\}
$$
and of multiplication by positive scalar $t$:
$$
tA=\{ta: a \in A\}.
$$
Outside of this, $bc(X)$ is a complete pseudometric space with respect to the Hausdorff distance.
$bc(X)$ is not a metric space because the Hausdorff distance does not distinguish between two sets that share the same closure. In order to make from $bc(X)$ a metric space one introduces the following equivalence relation: $A\approx B$ if $\rho_H(A,B) = 0$, and denotes $BC(X)$ the quotient space of the cone $bc(X)$ by the equivalence relation  $\approx$. This means that the elements of $BC(X)$ are equivalence classes of the form $[A] = \{B \in bc(X): A\approx B\}$.
The cone operations on $BC(X)$ are naturally defined as $[A] + [B] = [A + B]$, $t[A]=[tA]$, and the Hausdorff distance $\rho_H([A],[B]) := \rho_H(A,B)$     is a metric on  $BC(X)$.

With a little abuse of rigor, we will say that $BC(X)$ is the metric cone of all convex bounded subsets of the Banach space $X$, keeping in mind that equivalent sets are considered as equal ones.

One can avoid the  above difficulties in the definition of $BC(X)$ restricting the consideration to closed convex bounded subsets, but this leads to another inconvenience. Namely, the Minkowski sum of two closed sets is not necessarily closed, so if one wants to stay in the frames of closed sets, the sum of two sets should be defined as the closure of their Minkowski sum. The approach with closed sets is adopted for example in \cite{hess}. The both difficulties disappear if one restricts the considerations to compact convex sets. This is one of many reasons why very often one considers the integration of multifunctions only for functions that take compact convex values.

R{\aa}dstr\"om \cite{Radstr} remarked a very useful possibility of reducing the cone of closed convex bounded subsets of a space $X$ to a convex cone of elements in  another Banach spaces preserving the cone operations and the distance. Let us describe an explicit construction for such an embedding \cite[pp. 626 -- 627]{hess}.

\begin{definition} \label{def-radstr}
Denote $Z=Z(X)$ the linear space of all bounded continuous real functions on the unit sphere $S_{X^*}$ of the dual space and equip $Z$ with the norm
$$
\|f\| = \sup_{x^* \in S_{X^*}} |f(x^*)|.
$$
Define the R{\aa}dstr\"om embedding $R: BC(X) \to Z$ as follows: for every $A \in BC(X)$
$$
\left(R(A)\right)\left(x^*\right) =  \sup_{a \in A} x^*(a).
$$
\end{definition}
The definition is consistent because $\sup_{a \in A} x^*(a) = \sup_{a \in
\overline A} x^*(a)$.
It is plain that $R(tA+sB) = tR(A) + sR(B)$ for $A, B \in BC(X)$ and $t, s > 0$.
The equality $\rho_H(A,B) = \|R(A) - R(B)\|$ is a consequence of the Hahn-Banach theorem.

 The R{\aa}dstr\"om embedding enables us to transfer some definitions and facts about usual functions to the theory of multifunctions. A good example is the definition of the Bochner integral for multifunctions.

Let $(\Omega, \Sigma, \mu)$ be a measure space with $\mu(\Omega)<\infty$ and let $F\colon \Omega\to BC(X)$ be a multifunction. Then $R \circ F\colon \Omega\to Z(X)$ is an ordinary vector-valued function, and in this case the Bochner integral is a classical notion, defined through approximation by simple functions \cite{die-uhl-J}. One says that the multifunction $F$ is Bochner integrable with $\int_{\Omega} F d\mu = A$ if the ordinary function $R \circ F$ is  Bochner integrable with $\int_{\Omega} R \circ F d\mu = R(A)$.

This definition enables us to obtain ``for free'' the basic results on the Bochner integral of multifunctions. Nevertheless, it makes sense to ``decode'' the definition. One of the possibilities is the following one.

A function of the form
$$f= \sum_{i=1}^{n}U_i \mathds{1}_{\Delta_i} $$
with $U_i \in BC(X)$ and $\Delta_i \in \Sigma$ is called a simple multifunction. The Bochner integral of $f$ is defined as
$$\int_{\Omega}^{}f d\mu=\sum_{i=1}^{n}U_i \mu(\Delta_i). $$
A general multifunction $F\colon \Omega \to bc(X)$ is  Bochner integrable if the scalar function $t \mapsto \|F(t)\|$ has a Lebesgue integrable majorant and there are simple multifunctions $F_n\colon \Omega \to BC(X)$ such that $F_n(t) \xrightarrow[n \to \infty]{a.e.} F(t)$. These $F_n$ can be selected in such a way that all of them have a common Lebesgue integrable majorant, and then the Bochner integral of $F$ is
$$
\int_{\Omega} F d\mu = \lim_{n\to \infty}\int_{\Omega}^{}F_n d\mu.
$$

\subsection{R{\aa}dstr\"om embedding and random sets}

In our paper, we refer to the theory of random sets only for convex sets. The main reason is that, in the non-convex case, even the definition of the mathematical expectation is not so clear. Say, if a probability space consists of one point and the random set $F_1$ equals $\{-1, 1\}$ in that point, then any reasonable definition should say that the mathematical expectation  $E(F_1)$ is equal to $\{-1, 1\}$. Now, if a probability space consists of two points $t, \tau$ with $P(t) = P(\tau) = \frac12$, and the random set $F_2$ equals $\{-1, 1\}$ in both the points, then, for any intuitively reasonable definition of identically distributed random sets, $F_1$ and  $F_2$ should be  identically distributed. So, on the one hand, we should expect $E(F_2) = E(F_1) = \{-1, 1\}$, but on the other hand, $E(F_2)$ should be equal to
$$
F_2(t)P(t) + F_2(\tau)P(\tau) = \frac12 \{-1, 1\} +  \frac12 \{-1, 1\} = \{-1, 0, 1\}.
$$
This problem can be fixed if one restricts the considerations to non-atomic probability spaces, but these subtleties are inessential for the problem that we address in our paper.

The key to probabilistic concepts for convex random sets is the R{\aa}dstr\"om embedding. Let $(\Omega, \Sigma, \mu)$ be a probability space and $F\colon \Omega \to BC(X)$ be a multifunction. $F$ is said to be a random set if $R \circ F\colon \Omega \to Z(X)$ is measurable in the sense that $(R \circ F)^{-1}(U) \in \Sigma$ for every open subset $U \subset Z(X)$.  $F$ has mathematical expectation if it is Bochner integrable, and $E(F) := \int_{\Omega} F d\mu$. A collection $\{F_\alpha\}$ of random sets is independent, if the corresponding collection $\{R \circ F_\alpha\}$ of random vectors is independent. Two $BC(X)$-valued random sets $F_1, F_2$ (maybe on different probability spaces) are  identically distributed, if the corresponding random vectors $R \circ F_1, R \circ F_2$ are identically distributed.


\section{The full form of SLLN} \label{form1}

The forms of the SLLN that we mentioned in the introduction do not require that the random vectors $\Psi_n: \Omega \to X$ are  identically distributed, and for some spaces (namely, for B-convex ones)  the corresponding SLLN is valid and for others is not. In contrast to this, the SLLN for sequences of  independent identically distributed (i.i.d) random vectors works in every Banach space, and it can be easily transferred to i.i.d. random convex bounded sets with the help of the R{\aa}dstr\"om embedding (see \cite{artvit, purral}, where the result is stated for random compact sets, but the proof works in the general $BC(X)$ case as well).

It would be natural to apply the same idea to our forms of SLLN. Namely, for a  Banach space $X$, we can take a uniformly bounded independent sequence of random sets $\Psi_n: \Omega \to BC(X)$ having mathematical expectations and apply the R{\aa}dstr\"om embedding. We obtain a sequence  $R \circ \Psi_n: \Omega \to Z(X)$ of random vectors and for random vectors there could be a chance to get the result if  $\overline{R(BC(X))}$ would be B-convex. Unfortunately, this does not work because, according to  \cite[Theorem 4.2]{ArtKad}, for every infinite-dimensional $X$, $\overline{R(BC(X))}$ is not B-convex.

Below, we use some elements of the proof of \cite[Theorem 4.2]{ArtKad} in order to demonstrate that our full form of SLLN does not hold true in any  infinite-dimensional $X$. Moreover (Theorem \ref{thm-absSLLN}), it does not hold true even if we restrict our considerations to uniformly bounded sequences of finite-dimensional compact convex random sets and even for convergence in probability.

We are going to use the following two results. The first one is a lemma of combinatorial character.

\begin{lem}[{\cite[Lemma 4.1]{ArtKad}}] \label{combin}
Given a natural number $n$, consider a set $M$ with $\card(M) = 2^n$. There are $n$ subsets of $M$, say $T_1, \ldots, T_n$, such that whenever $W$ is a subset of $\{1, \ldots, n\}$,  then there is an element $m_W$ in $M$, such that $m_W \in T_j$ whenever $j \in  W$ and $m_W \not \in T_j$ when $j \not \in W$.
\end{lem}

The second one is the famous Auerbach's lemma, see for example \cite[Proposition 1.c.3]{L-Tz}.

\begin{theo} \label{thm-Auer}
Let $X$ be a Banach space of $\dim X \ge n$, then there is a collection of elements $\{e_k\}_{k=1}^n$ in $S_X$ such that
$$
\left\|\sum_{k=1}^n b_k e_k\right\| \ge \max_{1 \le k \le n} |b_k|
$$
for all collections of reals $\{b_k\}_{k=1}^n$.
\end{theo}

The last ingredient of the promised construction will be the following lemma extracted (with some modifications) from the demonstration of \cite[Theorem 4.2]{ArtKad}.

\begin{lem} \label{constr-old}
For every $n \in \N$ and every Banach space $X$ of $\dim X \ge 2^n$ there is a collection of  finite-dimensional compact convex subsets  $\{V_k\}_{k=1}^n$ of $B_X$ such that, for every collection of reals $\{a_k\}_{k=1}^n$,
\begin{equation} \label{0eqq111}
\left\|\sum\limits_{k=1}^{n} a_k R(V_k)\right\| \geq \frac{1}{2}\sum_{k=1}^n |a_k|,
\end{equation}
where $R$ is the  R{\aa}dstr\"om embedding. In particular, for every
choice of $\theta_k = \pm 1$,
\begin{equation} \label{0eqq112}
\left\|\sum\limits_{k=1}^{n} \theta_k R(V_k)\right\| \geq \frac{n}{2}.
\end{equation}

\end{lem}
\noindent \textit{Proof.}
At first, thanks to Auerbach's lemma  (Theorem \ref{thm-Auer}),  there is a collection of elements $\{e_k\}_{k=1}^{2^n} \subset S_X$ such that
$$
\left\|\sum_{k=1}^{2^n} b_k e_k \right\| \ge \max_{1 \le k \le 2^n} |b_k|
$$
for all collections of reals $\{b_k\}_{k=1}^{2^n}$.

We apply Lemma \ref{combin} with $M = \{1, \ldots, 2^n\}$ and get the corresponding subsets $T_1, \ldots, T_n \subset M$. Finally, we define the requested subsets  $\{V_k\}_{k=1}^n$ of $B_X$ as
$$
V_k = \conv\{e_i : i \in T_k\}.
$$
Now, let us fix $a_k \in \R$, $k = 1, \ldots , n$ and check the validity of \eqref{0eqq111}.
Let $W$ be the
subset of those $k \in \{1, \ldots, n\}$ where $\sign a_k = 1$. Denote $s = \sum_{k \in W} |a_k|$. Without loss
of generality $s \ge \frac{1}{2}\sum_{k=1}^n |a_k|$, otherwise we may work with the numbers $-a_k$ instead of $a_k$. By Lemma \ref{combin}, there
is an index $m \in M$ such that $m \in T_j$ whenever $j \in  W$ and $m \not \in T_j$ when $j \not \in W$. The corresponding
unit vector $e_m$ belongs to each $V_k$ with $k \in W$.
This implies that $se_m \in \sum_{k \in W} a_k V_k$.
On the other hand, $\|se_m - z\| \ge  s$ for all elements
$z \in \sum\limits_{i\notin W}|a_i| V_i$. Consequently,
\begin{align*}
\left\|\sum\limits_{i=1}^{n} a_i R(V_i)\right\| &= \left\| R\left(\sum\limits_{i\in W}|a_i|V_i\right) - R\left(\sum\limits_{i\notin W}|a_i|V_i\right) \right\|  \\
&= \hausd\left(\sum\limits_{i\in W}|a_i|V_i, \sum\limits_{i\notin W}|a_i|V_i\right)
 \ge \|se_m\| \ge \frac{1}{2}\sum_{k=1}^n |a_k|. \qed
\end{align*}

Let $(\psi_n)$ be an independent sequence of Bernoulli random variables on a probability space $(\Omega, \Sigma, P)$ taking values 0 and 1 with the probability  $\frac{1}{2}$.

\begin{theo} \label{thm-absSLLN}
For every infinite-dimensional Banach space $X$ there exists a sequence $(V_n)$ of  finite-dimensional compact convex subsets of the unit ball of $X$ and there is an $\eps > 0$ such that, for $\Psi_n = \psi_n \cdot V_n$,
\begin{align} \label{kein-SLLN}
\hausd\left(\frac{1}{4^n}\sum\limits_{i=1}^{4^n} \Psi_i(\omega), \frac{1}{4^n}\sum\limits_{i=1}^{4^n} E\left(\Psi_i(\omega)\right)\right) =  \nonumber \\ \hausd\left(\frac{1}{4^n}\sum\limits_{i=1}^{4^n} \psi_i(\omega) V_i, \frac{1}{4^{n}\cdot 2}\sum\limits_{i=1}^{4^n} V_i\right) \ge \eps
\end{align}
 in every point $\omega \in \Omega$ and for all $n \in \N$.
 In particular,
$$
\hausd\left(\frac{1}{N}\sum\limits_{i=1}^{N} \Psi_i(\omega), \frac{1}{N}\sum\limits_{i=1}^{N} E\left(\Psi_i(\omega)\right)\right) \not\rightarrow 0
$$
as $N \to \infty$.
\end{theo}

\begin{proof}
 We will demonstrate that $\eps = \frac{1}{16}$ works. Applying repeatedly Lemma \ref{constr-old} with $n=4,12, \ldots, 4^{m}-4^{m-1}, \ldots$ we select collections of sets $\{V_k\}_{k=4^{m-1}+1}^{4^m}$ in $B_X$ such that  for every choice of $\theta_i = \pm 1$
\begin{equation} \label{eqq111}
\left\|\sum\limits_{k=4^{m-1}+1}^{4^m} \theta_i R(V_i)\right\| \geq \frac{4^{m}-4^{m-1}}{2}.
\end{equation}
Then, taking in account that $\psi_i(\omega)- \frac{1}{2} = \pm \frac{1}{2}$ for all $\omega \in \Omega$, we obtain
\begin{align*}
\hausd\left(\frac{1}{4^n}\sum\limits_{i=1}^{4^n} \psi_i(\omega) V_i, \frac{1}{4^{n}\cdot 2}\sum\limits_{i=1}^{4^n} V_i\right) = \frac{1}{4^n}\left\| R\left(\sum\limits_{i=1}^{4^n} \psi_i(\omega) V_i\right) -R\left(\frac{1}{2}\sum\limits_{i=1}^{4^n} V_i\right)\right\| \\
= \frac{1}{4^n}\left\| \sum\limits_{i=1}^{4^n} \psi_i(\omega) R\left(V_i\right) - \frac{1}{2}\sum\limits_{i=1}^{4^n} R\left(V_i\right)\right\| = \frac{1}{4^n}\left\| \sum\limits_{i=1}^{4^n} \left(\psi_i(\omega)- \frac{1}{2}\right) R\left(V_i\right) \right\| \\
\ge \frac{1}{4^n}\left\| \sum\limits_{i=4^{n-1}+1}^{4^n} \left(\psi_i(\omega)- \frac{1}{2}\right) R\left(V_i\right) \right\| - \frac{1}{4^n}\left\| \sum\limits_{i=1}^{4^{n-1}} \left(\psi_i(\omega)- \frac{1}{2}\right) R\left(V_i\right) \right\| \\
\ge \frac{1}{4^n}\frac{1}{2}\frac{4^{n}-4^{n-1}}{2} - \frac{1}{4^n}\frac{1}{2}4^{n-1} = \frac{1}{4^n}\frac{1}{4}\left(4^{n}-4^{n-1} - 2\cdot4^{n-1} \right) = \frac{1}{4^n}\frac{1}{4}4^{n-1}
 = \frac{1}{16}.
\end{align*}
\end{proof}


\section{The full form of SLLN in a finite-dimensional space} \label{form1fd}

The results of this section are not new. They follow, for example, from \cite[Theorem 1.19]{Molc}. We just present here a short independent proof.

\begin{theo} \label{thm-SLLN-form1-fd++}
Let $X$ be a Banach space,  $(\Omega, \Sigma, \mu)$ be a probability space and $\Psi_n: \Omega \to BC(X)$ be an independent sequence of random sets  with mathematical expectations  $E(\Psi_n) = U_n$, and let additionally all the values of all $\Psi_n$ lie in a common convex compact subset $W \subset X$. Then
\begin{equation} \label{eq0f1-fd}
\rho_H\left(\frac{1}{n}\sum\limits_{i=1}^{n} \Psi_i(\omega), \frac{1}{n}\sum\limits_{i=1}^{n} U_i \right) \xrightarrow[n \to \infty]{a.e.} 0.
\end{equation}
\end{theo}

\begin{proof} Recall that the space $Z$ from the Definition \ref{def-radstr} of  the R{\aa}dstr\"om embedding $R: BC(X) \to Z$ is the space of all continuous real functions on $S_{X^*}$ equipped with the uniform norm.

The collection $BC(W)$ of all non-empty convex closed subsets of $W$ is compact in the Hausdorff metric. This means that all the elements  $R(\Psi_i(\omega)), R(U_i)$ and their convex combinations lie in the compact convex subset $R(BC(W)) \subset Z$ and all pairwise differences of such elements lie in the compact subset $K:=R(BC(W))-R(BC(W)) \subset Z$.

Consider the mapping $T: S_{X^*} \to C(K)$, $(Tx^*)(f):= f(x^*)$. The space of continuous functions on a metric compact is separable, so in particular $ C(K)$ is separable, and $T(S_{X^*})$ is separable as well. Let $G=\{x_k^*\}_{k=1}^\infty \subset X^*$ be such a countable set that $T(G)$ is dense in $T(S_{X^*})$. Then $G$ separates points of $K$ in the sense that for every two different functions $f, g \in K$ there is $k \in \N$ such that $f(x_k^*) \neq g(x_k^*)$.

Denote $\nu$ the topology on $K$  generated by the norm and denote $\tau$ the topology in $K$ of pointwise convergence at all $\{x_k^*\}_{k=1}^\infty$. Since $K$ is compact in the topology $\nu$, and $\tau$ is a weaker Hausdorff topology, we have that $\nu = \tau$. In particular, every sequence in $K$ that converges at all $\{x_k^*\}_{k=1}^\infty$ is convergent uniformly.

For each fixed $x^* \in S_{X^*}$ the values $\left(R(\Psi_i(\omega))\right)\left(x^*\right)$ form a uniformly bounded  independent sequence of random variables on $\Omega$ with $E\left(\left(R(\Psi_i(\cdot))\right)\left(x^*\right)\right) = (R(U_i))(x^*)$. The ordinary SLLN for random variables says that for each $x^* \in S_{X^*}$
\begin{equation*} \label{eq1f1-fd-}
\left|\frac{1}{n}\sum\limits_{i=1}^{n} R\left(\Psi_i(\omega)\right)\left(x^*\right) - \frac{1}{n}\sum\limits_{i=1}^{n} R\left(U_i\right)(x^*)\right| \xrightarrow[n \to \infty]{a.e. \ in \ \omega} 0.
\end{equation*}
This means that for every $k \in \N$ there is a set $A_k \in \Sigma$ with $\mu(A_k) =0$ such that for every $\omega \in \Omega \setminus A_k$
\begin{equation} \label{eq1f1-fd}
\left|\left(R\left(\frac{1}{n}\sum\limits_{i=1}^{n} \Psi_i(\omega)\right) - R\left(\frac{1}{n}\sum\limits_{i=1}^{n} U_i\right)\right)(x_k^*)\right| \xrightarrow[n \to \infty]{} 0.
\end{equation}
Denote $A = \bigcup_{k=1}^\infty A_k$, $\mu(A) =0$. Then for every $\omega \in \Omega \setminus A$ the condition \eqref{eq1f1-fd} works for all $k \in \N$ simultaneously. As we remarked before, this implies that  for every $\omega \in \Omega \setminus A$

$$
\left\|R\left(\frac{1}{n}\sum\limits_{i=1}^{n} \Psi_i(\omega)\right) - R\left(\frac{1}{n}\sum\limits_{i=1}^{n} U_i\right)\right\|_{C(S_{X^*})}  \xrightarrow[n \to \infty]{} 0.
$$
Since $R$ is an isometry, this gives us the condition \eqref{eq0f1-fd} we want.
\end{proof}

As an immediate consequence we obtain the promised full form of SLLN for finite-dimensional spaces.

\begin{theo} \label{thm-SLLN-form1-fd}
Let $X$ be a finite-dimensional Banach space,  $(\Omega, \Sigma, \mu)$ be a probability space and $\Psi_n: \Omega \to BC(X)$ be a uniformly bounded independent sequence of random sets  with mathematical expectations  $E(\Psi_n) = U_n$. Then
\begin{equation} \label{eq0f1-fd}
\rho_H\left(\frac{1}{n}\sum\limits_{i=1}^{n} \Psi_i(\omega), \frac{1}{n}\sum\limits_{i=1}^{n} U_i \right) \xrightarrow[n \to \infty]{a.e.} 0.
\end{equation}
\end{theo}


\section{The reduced form of SLLN} \label{form3}

The reduced form of SLLN for random sets that we demonstrate in Theorem \ref{thm-form3} at the end of this section, is relatively simple in the separable case. So, our main effort here is concentrated on the following ``Separable reduction lemma''.

\begin{lem}\label{sep_sub}
Let $F\colon \Omega \to BC(X)$ be a Bochner integrable multifunction.
Then there is a separable subspace $Y \subset X$ and a Bochner integrable multifunction $\widetilde{F}\colon \Omega \to BC(Y)$, such that  $\widetilde{F}(t)\subset F(t)$ and $\diam\widetilde{F}(t)= \diam F(t)$ for almost all $t \in \Omega$.
\end{lem}

The proof will be done through a chain of propositions of increasing complexity.

\begin{propos}\label{prop-ch1}
Let $X$ be a Banach space, $U, V \in b(X)$ and let
$B \subset V$ be a countable subset. Then there is a countable subset $A \subset U$ such that $\bar{\rho}_H(A,B) \le \bar{\rho}_H(U,V)$.
\end{propos}
\begin{proof}
Let us enumerate $B$ as $\{b_1, b_2, \ldots\}$. According to the definition of $\bar{\rho}_H(U,V)$, for every $v \in V$ and every $\eps > 0$ there is a $u \in U$ such that $\|u - v\| < \bar{\rho}_H(U,V) + \eps$. In particular, for each of $b_n$ there is a $u_{n,m} \in U$ such that
$$
\|u_{n,m} - v_n\| < \bar{\rho}_H(U,V) + \frac{1}{m}.
$$
Then $A = \{u_{n,m}\}_{n,m \in \N}$ is the set we need.
\end{proof}

\begin{propos}\label{prop-ch2}
Let $X$ be a Banach space, $U \in b(X)$ $V_m \in b(X)$, $m = 1,2, \ldots$ and let $A \subset U$, $B_m \subset V_m$, $m = 1,2, \ldots$, be countable subsets. Then there is a countable subset $\widetilde A$ such that $A \subset \widetilde A \subset U$ and $\bar{\rho}_H(\widetilde A,B_m) \le \bar{\rho}_H(U,V_m)$ for all $m = 1,2, \ldots$.
\end{propos}
\begin{proof}
Applying successively Proposition \ref{prop-ch1}, we can build countable subsets $A_m \subset U$ such that
$$
\bar{\rho}_H(A_{m},B_m) \le \bar{\rho}_H(U,V_m), \quad m= 1,2, \ldots
$$
It remains to define the required $\widetilde A$ as $\widetilde A = A \cup \left(\bigcup_{m \in \N}A_m\right)$.
\end{proof}

\begin{propos}\label{prop-ch3}
Let $X$ be a Banach space and let $U_n \in b(X), n\in \N$, be a countable family of sets. Let $A_n \subset U_n$, $n\in \N$, be countable subsets. Then there are countable sets $\widetilde A_n$,  $n\in \N$, such that
$$A_n\subset \widetilde A_n \subset U_n$$
and, for all $n,m\in \N$,
\begin{equation} \label{prop-ch3}
\rho_H(\widetilde A_n, \widetilde A_m)\leq \rho_H(U_n,U_m).
\end{equation}
\end{propos}
\begin{proof}
Let us define sets $A_{n,1}=A_n$. After that, we build a table $A_{n,k}$ of countable subsets in such a way that, for each $n, m, k \in \N$,
\begin{equation} \label{prop-ch3-eq01}
A_{n,k} \subset A_{n,k+1} \subset U_n \textrm{ and } \bar{\rho}_H(A_{n,k+1},A_{m,k}) \le \bar{\rho}_H(U_n,U_m).
\end{equation}
(We do this, applying recurrently Proposition \ref{prop-ch2} with $A = A_{n,k}$,  $U=U_n$, $V_m=U_m$,  and $B_m = A_{m,k}$, $m=1,2, \ldots$. Then  $\widetilde A$ from Proposition \ref{prop-ch2} is $A_{n,k+1}$ we need).

Now we define $\widetilde A_n = \bigcup_{k \in \N} A_{n,k}$. Then, evidently, $A_{n} \subset \widetilde A_n \subset U_n$. Let us demonstrate that
\begin{equation} \label{prop-ch3-eq02}
\bar{\rho}_H(\widetilde A_n, \widetilde A_m) \le \bar{\rho}_H(U_n,U_m).
\end{equation}
Indeed, for every $b \in \widetilde A_m = \bigcup_{k \in \N} A_{m,k}$, there is $k \in \N$ such that $b \in A_{n,k}$. Then, for every $\eps > 0$, we can apply \eqref{prop-ch3-eq01} and get $a \in A_{n,k+1} \subset \widetilde A_n$ with $\|a - b\| < \bar{\rho}_H(U_n,U_m) + \eps$. This proves \eqref{prop-ch3-eq02}.
Switching the roles of $n$ and $m$ in \eqref{prop-ch3-eq02}, we get
\begin{equation} \label{prop-ch3-eq03}
\bar{\rho}_H(\widetilde A_m, \widetilde A_n) \le \bar{\rho}_H(U_m,U_n).
\end{equation}
Then, inequalities  \eqref{prop-ch3-eq02} and  \eqref{prop-ch3-eq03} together give us the required condition  \eqref{prop-ch3}.
\end{proof}

\noindent\textit{Proof of Lemma \ref{sep_sub}}.
Due to the definition of Bochner integral, there are simple multifunctions $F_n\colon \Omega \to BC(X)$, $F_n=\sum_k U_{n,k} \eins_{\Delta_{n,k}}$ such that $F_n(t) \xrightarrow[n \to \infty]{a.e.} F(t)$.  Let us choose countable sets $A_{n,k}\subset U_{n,k}$ with $\diam A_{n,k}= \diam U_{n,k}$. Applying Proposition \ref{prop-ch3} to the countable collection of sets $(U_{n,k})$ we get corresponding countable subsets $\widetilde A_{n,k}$, $A_{n,k} \subset \widetilde A_{n,k} \subset U_{n,k}$ with
$$
\rho_H(\widetilde A_{n,k}, \widetilde A_{m,j})\leq \rho_H(U_{n,k},U_{m,j}).
$$
for all $n, k, m, j$. Now we can define the required separable subspace $Y$ as $Y = \spn \bigcup_{n, k} \widetilde A_{n,k}$  and
 $\widetilde F_n \colon \Omega \to BC(Y)$,  $\widetilde F_n=\sum_k \widetilde A_{n,k} \eins_{\Delta_{n,k}}$.  Due to the Cauchy criterion, the sequence of $F_n$ converges a.e. to some Bochner integrable function which is the desired $\widetilde{F}$.
\qed

Recall that a subset $W \subset X^*$ is called \textit{total} if for every $x\in X \setminus \{0\}$ there is $x^* \in W$ such that $x^*(x) \neq 0$. It is well-known that for every separable $X$ there is a countable total subset $W \subset X^*$, see for example  \cite[Section 17.2.4]{kadets-fa}.

\begin{theo}\label{one_point}
Let  $F\colon \Omega \to BC(X)$ be a Bochner integrable multifunction with
$$\int_{\Omega}^{}F d\mu =\{0\},$$
then for almost all $t\in \Omega$ the set $F(t)$ contains only one point.
\end{theo}
\begin{proof}
{\bf Separable case.} Consider a countable total set $\{x_n^*\}_{n=1}^{\infty} \subset X^*$. Then for every $n\in \N$ the composition $x_n^*\circ F$ is Bochner integrable multifunction. Then $x_n^*\circ F(t) \approx [f_{1,n}(t), f_{2,n}(t)]$, where $f_{1,n}(t) = \inf x_n^*(F(t)), f_{2,n}(t)=\sup x_n^*(F(t))$. In these notations we have
$$\int_{\Omega}^{}x_n^*\circ F(t)d\mu = \left[\int_{\Omega}^{}f_{1,n}(t)d\mu, \int_{\Omega}^{}f_{2,n}(t)d\mu\right]=\{x_n^*(0)\}=\{0\}.$$
Together with the fact $f_{2,n}-f_{1,n}\geq 0$ the previous equality implies $f_{1,n}=f_{2,n}$ almost everywhere. That is, for every $n\in\N$ there is a set $A_n\in \Sigma$ with $\mu(A_n)=0$ such that for every $t\in \Omega\setminus A_n$ the equality holds $f_{1,n}(t)=f_{2,n}(t)$. Denote $A=\cup_{n=1}^{\infty}A_n$. Observe $\mu(A)=0$ and for every $n\in \N, t\in \Omega\setminus A$ we have $f_{1,n}(t)=f_{2,n}(t)$. This means $F(t)$ consists of one point. In fact, suppose there are $x_1, x_2 \in F(t)$ with $x_1\neq x_2$. Then there is $n\in \N$ such that $x_n^*(x_1-x_2)\neq 0$, so $x_n^*(x_1)\neq x_n^*(x_2)$, which is a contradiction.\\
{\bf Non-separable case.} The proof in this case follows directly from Lemma \ref{sep_sub} and the proof for the separable case.
\end{proof}

A.~Beck states his ``SLLN for random vectors in B-convex spaces'' in the reduced form. So, the validity of this form of SLLN for random sets in B-convex spaces is a natural guess. Indeed, the following result holds true:

\begin{theo} \label{thm-form3}
If $X$ is a B-convex Banach space, then for every independent sequence of random sets $\Psi_n: \Omega \to BC(X)$ with mathematical expectations equal to $\{0\}$ and uniformly bounded variances one has
$$
\left\|\frac{1}{n}\sum\limits_{i=1}^{n} \Psi_i(\omega)\right\| \xrightarrow[n \to \infty]{a.e.} 0.
$$
\end{theo}
\begin{proof}
Theorem \ref{one_point} states that such random sets are almost everywhere single-valued, that is they are in fact not random sets but random vectors, and the Beck's theorem is applicable.
\end{proof}


\section{The full form with finite-dimensional expectations and the intermediate form of SLLN} \label{form2}

In this section we extend Theorem \ref{thm-SLLN-form1-fd} to the case of B-convex space $X$ under the additional condition that expectations of the random sets lie in one the same finite-dimensional subspace of $X$.

That gives us the validity of the intermediate form of SLLN for random sets in B-convex spaces for the case of finite-dimensionality of the common mathematical expectation of the random sets in question.

For the proof we need three main ingredients: the validity of the SLLN in its full form in the case of finite-dimensional space $X$ (Theorem \ref{thm-SLLN-form1-fd}), Theorem \ref{one_point} on random sets whose expectation is a singleton, and the reduced form of SLLN in B-convex spaces (Theorem \ref{thm-form3}).

\begin{theo} \label{thm-SLLN-form1-fdexp}
If $X$ is a B-convex Banach space,  $(\Omega, \Sigma, \mu)$ is a probability space and $U_n \in BC(X)$, $n=1,2, \ldots$, lie in a finite-dimensional subspace $Y \subset X$, then for every uniformly bounded independent sequence of random sets $\Psi_n: \Omega \to BC(X)$ with  $E(\Psi_n) =U_n$ one has
$$
\rho_H\left(\frac{1}{n}\sum\limits_{i=1}^{n} \Psi_i(\omega), \frac{1}{n}\sum\limits_{i=1}^{n} U_i\right) \xrightarrow[n \to \infty]{a.e.} 0.
$$
\end{theo}

\begin{proof}
Thanks to finite-dimensionality, $Y$ is complemented in $X$. This means that
there is a continuous linear projector $P: X \to Y$ and $X = Y \oplus \ker P$.
Denote $Q(x) = x - P(x)$, then $Q$ is a continuous linear projector of $X$ on  $\ker P$, and $\ker Q = Y$.

Consider auxiliary  random sets $\Phi_n: \Omega \to BC(\ker P)$: $\Phi_n(\omega) := Q(\Psi_n(\omega))$. For these random sets,
$$
E(\Phi_n) = Q(E(\Psi_n)) = Q(U_n) = \{0\}.
$$
Applying Theorems \ref{one_point} and \ref{thm-form3} (or, in fact, the Beck's theorem) we see that for almost all $\omega \in \Omega$ the set $\Phi_n(\omega)$ is a singleton, which means that there exists a random vector
$\phi_n: \Omega \to \ker P$ such that
$$
Q(\Psi_n(\omega)) = \Phi_n(\omega) \stackrel{a.e.}{=} \{\phi_n(\omega)\},
$$
and that
\begin{equation} \label{eq0f1-fd++}
\left\|\frac{1}{n}\sum\limits_{i=1}^{n} \{\phi_n(\omega)\}\right\| \xrightarrow[n \to \infty]{a.e.} 0.
\end{equation}
Moreover, for the independent random sets $\Lambda_n := \Psi_n(\omega) - \{\phi_n(\omega)\}$ we have
$$
Q\left( \Lambda_n(\omega)\right) =  \{\phi_n(\omega)\}- \{\phi_n(\omega)\}= \{0\}
$$
a.e., which means that $\Lambda_n(\omega) \subset Y$ a.e.
Taking in the account that $E(\Lambda_n) = E(\Psi_n) - E(\{\phi_n\}) = E(\Psi_n) - \{0\} =  E(\Psi_n) = U_n$ and applying Theorem \ref{thm-SLLN-form1-fd} we obtain that
\begin{equation} \label{eq0f1-fd-compl}
\rho_H\left(\frac{1}{n}\sum\limits_{i=1}^{n}\Lambda_n(\omega) , \frac{1}{n}\sum\limits_{i=1}^{n} U_i \right) \xrightarrow[n \to \infty]{a.e.} 0.
\end{equation}
Putting \eqref{eq0f1-fd-compl} and \eqref{eq0f1-fd++} together we obtain what we need:
\begin{align*}
&\rho_H\left(\frac{1}{n}\sum\limits_{i=1}^{n} \Psi_i(\omega), \frac{1}{n}\sum\limits_{i=1}^{n} U_i\right) \le \rho_H\left(\frac{1}{n}\sum\limits_{i=1}^{n} \Psi_n(\omega) , \frac{1}{n}\sum\limits_{i=1}^{n}\Lambda_n(\omega) \right) \\&+ \rho_H\left(\frac{1}{n}\sum\limits_{i=1}^{n}\Lambda_n(\omega), \frac{1}{n}\sum\limits_{i=1}^{n} U_i \right) \\ &=\left\|\frac{1}{n}\sum\limits_{i=1}^{n} \{\phi_n(\omega)\}\right\| + \rho_H\left(\frac{1}{n}\sum\limits_{i=1}^{n}\Lambda_n(\omega), \frac{1}{n}\sum\limits_{i=1}^{n} U_i \right) \xrightarrow[n \to \infty]{a.e.} 0.
\end{align*}
\end{proof}

\begin{cor} \label{thm-SLLN-form2}
Let $X$ be a B-convex Banach space,  $(\Omega, \Sigma, \mu)$ be a probability space, and let $U \in BC(X)$ be a subset with $\dim \spn (U) < \infty$, then for every uniformly bounded independent sequence of random sets $\Psi_n: \Omega \to BC(X)$ with mathematical expectations equal to $U$ one has
$$
\rho_H\left(\frac{1}{n}\sum\limits_{i=1}^{n} \Psi_i(\omega), U\right) \xrightarrow[n \to \infty]{a.e.} 0.
$$
\end{cor}


\section{Concluding remarks and open questions} \label{questions}

As we already mentioned at the beginning of Section \ref{form1}, the SLLN for sequences of independent identically distributed (i.i.d) random vectors works in every Banach space, and it can be transferred to i.i.d. random convex bounded sets with the help of the R{\aa}dstr\"om embedding \cite{artvit, purral}.  In the  case of i.i.d. convex random sets there is no difference between the full and the reduced forms of of SLLN, because both of them are valid in every space. The situation complicates if one tries to generalize the SLLN to non-convex sets. This requires additional accuracy with the definitions, but such generalizations are possible with the help of convexification argument that work in every finite-dimensional space  \cite{artvit}, or in every Banach space, if the considerations are restricted to compact sets \cite{arthan}, see \cite[pp. 653-654]{hess} for more information and references. It was recently remarked that the convexification works for all bounded sets in a space $X$ if and only if the space $X$ is B-convex \cite{ArtKad}. This convexification argument can be applied to our Theorem \ref{thm-SLLN-form1-fdexp} which gives the following result for free:

\begin{theo} \label{thm-SLLN-form1-nonc}
Let $X$ be a B-convex Banach space,  $(\Omega, \Sigma, \mu)$ be a probability space, and let  $\Psi_n: \Omega \to 2^X \setminus \{\emptyset\}$ be a uniformly bounded sequence of multifunctions such that $\conv (\Psi_n(\cdot))$ form an independent sequence of random sets  with  $E(\conv (\Psi_n))$, $n=1,2, \ldots$, lying in the same finite-dimensional subspace of $X$, then
$$
\rho_H\left(\frac{1}{n}\sum\limits_{i=1}^{n} \Psi_i(\omega), \frac{1}{n}\sum\limits_{i=1}^{n} E(\conv (\Psi_n))\right) \xrightarrow[n \to \infty]{a.e.} 0.
$$
\end{theo}

The SLLN for non-identically distributed sets (like in our paper) was addressed previously by several authors, but the main results require additional compactness restrictions and are proved for weaker types of convergence than the convergence with respect to the Hausdorff distance, see \cite[Corollary 3.4]{Hiai} as a good example.

There are several natural questions that we are not able to answer.

\begin{prob}
Does the full form of SLLN for random sets work if $X$ is B-convex and all the values of random sets are finite-dimensional of the same dimension?
\end{prob}

\begin{prob}
Does the full form of SLLN for random sets work if $X$ is B-convex and  all the random sets have finite-dimensional expectations of the same dimension?
\end{prob}

\begin{prob} \label{prob 3}
Does the intermediate form of SLLN for random sets work if  $X$ is B-convex, without any additional restrictions on the common mathematical expectation of the random sets? In other words, is it true that if $X$ is a B-convex Banach space, and  $\Psi_n: \Omega \to BC(X)$ is a uniformly bounded independent sequence of random sets with the same mathematical expectation $U$ one has
$$
\rho_H\left(\frac{1}{n}\sum\limits_{i=1}^{n} \Psi_i(\omega), U\right) \xrightarrow[n \to \infty]{a.e.} 0?
$$
\end{prob}

We had a feeling that a counterexample to Problem \ref{prob 3} can be constructed based on our ideas from Section \ref{form1}, composing that example from parts taken from Theorem \ref{thm-absSLLN} in a way to get equal  mathematical expectations loosing finite-dimensionality that we have in Theorem \ref{thm-absSLLN}. Unfortunately, our attempts to perform such a construction failed.

\begin{prob} \label{prob 4}
Does the intermediate form of SLLN work if $X$ is B-convex and the common mathematical expectation of the random sets is a compact set?
\end{prob}
Here we vote for the positive answer \smiley, taking in account good approximations of compact sets by finite-dimensional ones. There are several obstacles that we did not overcome. One of them we formulate in the next problem.

\begin{prob} \label{prob 5}
Does the following small perturbation of Theorem \ref{thm-form3} hold true?

-- \emph{If $X$ is a B-convex Banach space, then for every $\eps > 0$ there is a $\delta > 0$ such that for every $U \in  BC(X)$, $U \subset \delta B_X$ and every independent sequence of random sets $\Psi_n: \Omega \to BC(X)$ with $\Psi_n(\omega) \subset B_X$  and  $E(\Psi_n)=U$, $n = 1, 2, \ldots$, one has a.e.
$$
\limsup_{n \to \infty}\left\|\frac{1}{n}\sum\limits_{i=1}^{n} \Psi_i(\omega)\right\| \le \eps.
$$}
\end{prob}

\section*{Acknowledgments}

The first author acknowledges support by the KAMEA program administered by the Ministry of Absorption, Israel.
The second author was supported by the Akhiezer Foundation Grant.

The project was initiated during the visit of the second author to Holon Institute of Technology (HIT), Israel. The second author is grateful to HIT for the support and hospitality.

\bibliographystyle{amsplain}
\end{document}